\documentclass[a4paper, 11pt, leqno, twoside]{amsart}
\usepackage{hyperref}

\usepackage{float}

\usepackage[english]{babel} 

\usepackage{amsmath,amssymb,latexsym,amsthm,enumerate,epsfig,graphicx}
\emergencystretch=\maxdimen
\newtheorem{definition}{Definition}[section]
\newtheorem{theorem}{Theorem}[section]
\newtheorem{proposition}{Proposition}[section]
\newtheorem{lemma}{Lemma}[section]

\newtheorem{remark}{Remark}[section]
\numberwithin{equation}{section}

\newif\ifcomment \commentfalse
\def\commentON{\commenttrue}

\long\outer\def\BC#1\\EC{\ifcomment \sloppy \par \# \ldots\dotfill
{\em #1} \dotfill \# \par \fi } \commentON

\newcommand{\remove}[1]{}

\newenvironment{vardesc]}[1]{%
\settowidth{\parindent}{#1: \ }
\makebox{#1:}}{}

\newcommand{\R}{\mathbb{R}}

\begin{document} 
\title[Global existence and asymptotic stability to a model of biofilms]{Global existence and asymptotic stability of smooth solutions to a fluid dynamics model of biofilms in one space dimension}

\author{R. Bianchini$^{\diamond}$ \& R. Natalini$^{\star}$}

\thanks{$^{\diamond}$ Dipartimento di Matematica, Universit\`a degli Studi di Roma "Tor Vergata", via della Ricerca Scientifica 1, I-00133 Rome, Italy.}
\thanks{$^{\star}$ Istituto per le Applicazioni del Calcolo "M. Picone", Consiglio Nazionale delle Ricerche, via dei Taurini 19, I-00181 Rome, Italy.}


\maketitle

\begin{abstract}
In this paper, we present an analytical study,  in the one space dimensional case, of the fluid dynamics  system proposed in \cite{cdnr}  to model the formation of biofilms.  After  showing the hyperbolicity of the system, we show that, in a open neighborhood of the physical parameters, the system is totally dissipative near its unique non vanishing equilibrium point. Using this property, we are able to prove existence and uniqueness of global smooth solutions to the Cauchy problem on the whole line for small perturbations of this equilibrium point and the solutions are shown to converge exponentially in time at the equilibrium state.\\ \\
\end{abstract}


\section{Introduction}
\label{intro} 
A biofilm is a complex gel-like aggregation of microorganisms like bacteria, algae, protozoa and fungi. They stick together, attach to a surface and embed themselves in a self-produced extracellular matrix of polymeric substances, called EPS. \\
\indent In this paper, we study a fluid dynamics model, introduced in \cite{cdnr}, to describe the space-time growth of  biofilms. This model was built in the framework of mixture theory, see  \cite{raj95} or \cite{astanin}, and conserves the finite speed of propagation of the fronts. For simplicity reasons, the model describes a biofilm in which there is just one species of microorganisms, or better, all species are lumped together, but it can be extended to other situations. It has been derived starting from the equations for mass and momentum conservation, and some physical constraints  and assumptions about the behavior of the biological aggregates and their interaction with the surrounding liquid. Here we assume that the complex structure of biofilms is described by four different phases: bacteria $B(x, t)$, extracellular matrix EPS $E(x, t)$, dead cells $D(x, t)$ and a liquid phase $L(x,t)$. The quantities $B, E, D, L$ are the volume fraction of each component, then $B, E, D, L \in [0,1]$. Since we are dealing with a one dimensional model, then we have $x \in \mathbb{R}$, $t>0$. For simplicity reasons, we assume that $B, E, D$ have the same transport velocity $v_{S}$. The reaction terms are indicated by $\Gamma_{\Phi}$, with $\Phi = B, E, D, L$. Imposing the total balance of mass and momentum for each phase $\Phi$, we can write the model, see \cite{cdnr} for all details. This model was originally proposed in all space dimensions, and in the present case of one space dimension, is given by a system of six partial differential equations, which read:
\begin{equation}
\left\{
\label{completo_completo}
\begin{aligned}
& \partial_{t}B+\partial_{x}(B{v}_{S}) =\Gamma_{B}, \\
& \partial_{t}E+\partial_{x}(E{v}_{S}) =\Gamma_{E},\\
& \partial_{t}D+\partial_{x}(D{v}_{S}) =\Gamma_{D}, \\
& \partial_{t}L+\partial_{x}(L{v}_{L}) = \Gamma_{L}, \\
& \partial_{t}((1-L)v_{S}) + \partial_{x}((1-L)v_{S}^{2}) = -(1-L)\partial_{x}P-\gamma \partial_{x}(1-L) \\
&~~~~~~~~~~~~~~~~~~~~~~~~~~~~~~~~~~~~~~~~~~~+(M-\Gamma_{L})v_{L}-Mv_{S};  \\
& \partial_{t}(Lv_{L}) + \partial_{x}(Lv_{L}^{2}) = -L\partial_{x}P-(M-\Gamma_{L})v_{L}-Mv_{S}. \\
\end{aligned}
\right.
\end{equation}

To reformulate our model in a more suitable form, we assume the following volume constraint:
\begin{equation}
\label{satur}
L=1-(B+E+D),
\end{equation}
that is the assumption that the mixture is saturated and no empty space is left. In addition to the balance mass of each component, we also have the total conservation of the mass of the mixture by the following assumption:
\begin{equation}
\label{conservazionemassatotale}
\Gamma_{B}+\Gamma_{E}+\Gamma_{D}+\Gamma_{L}=0.
\end{equation}
The mass constraint in (\ref{conservazionemassatotale}) states that the mixture is closed, i.e. there is no net production of mass for the mixture. According to \cite{cdnr}, the reaction terms are given by:
\begin{equation}
\Gamma_{B}=K_{B}BL-K_{D}B;
\end{equation}
\begin{equation}
\Gamma_{D}=\alpha K_{D}B-K_{N}D;
\end{equation}
\begin{equation}
\Gamma_{E}=K_{E}BL-\epsilon E.
\end{equation}

The birth of new cells at a point depends on the quantity of liquid available in the neighborhood of the point, that is why the birth term in $\Gamma_{B}$ is a
product between the volume fraction $B$ of active cells and the volume fraction $L$ of liquid. In this way, the mass production term $\Gamma_{B}$ is the difference between a birth term and a death term, where the second is proportional to the fraction $B$ of bacteria, with rate $k_{D}$. The death term in the expression of $\Gamma_{B}$ gives rise to a creation term of the mass exchange rate for dead cells $\Gamma_{D}$, with a proportional coefficient $\alpha$, since a part of the active cells goes into liquid when the cell dies. In $\Gamma_{D}$, we also find a natural decay of dead cells with a constant decay rate $k_{N}$.
The EPS is produced by active cells in presence of liquid and therefore the production term will
be proportional to $BL$, where $k_{E}$ is the growth rate of EPS. There is also a natural decay of EPS with rate $\epsilon$.
To conclude this explanation about the mass exchange terms, we choose $\Gamma_{L}$ in order to enforce condition (\ref{conservazionemassatotale}). See again \cite{cdnr} for more details. \\

Let us now further simplify the equations of system (\ref{completo_completo}). First, by the volume constraint, the equation for the fraction $L$ is no more necessary. Summing the equations for $B, E, D, L$ in (\ref{completo_completo}) and using the volume constraint (\ref{satur}), we obtain 
\begin{equation}
\label{divergence_free}
 \partial_{x}((1-L)v_{S}+Lv_{L}) = 0,
\end{equation}
which establishes that the derivative of the average of the hydrodynamic velocity vanishes. Since we are in one space dimension and since all phases $B, E, D, L$ and velocities $v_{S}, v_{L}$ vanish at infinity, by using equation (\ref{divergence_free}) we can express the unknown liquid phase velocity in function of the solid phase velocity, and then we can eliminate the equation for the liquid phase from system (\ref{completo_completo}). Precisely, we have:
\begin{equation}
\label{vel.}
v_{L} = \dfrac{L-1}{L}v_{S}.
\end{equation}
Summing up the fifth and the sixth equation of system (\ref{completo_completo}) and using equality (\ref{vel.}), we have an explicit expression of the spatial derivative of the pressure term $P$, which solves
\begin{equation}
\label{Pressure}
\partial_{x}P = -\gamma \partial_{x}(1-L) - \partial_{x}((1-L)v_{S}^{2}+Lv_{L}^{2}).
\end{equation}
Then, substituting equation (\ref{vel.}) in (\ref{Pressure}), we have
\begin{equation}
\label{Pressure_New}
\partial_{x}P = - \gamma \partial_{x}(1-L) -\partial_{x}\Bigg(\dfrac{1-L}{L}v_{S}^{2}\Bigg).
\end{equation}
After some calculations, we obtain an equation for the solid phase velocity $v_{S}$, and, finally, system (\ref{completo_completo}) becomes a system of only four equations, which are 
\begin{equation}
\left\{
\label{completo}
\begin{aligned}
& \partial_{t}B+\partial_{x}(B{v}) =\Gamma_{B}, \\
& \partial_{t}E+\partial_{x}(E{v}) =\Gamma_{E},\\
& \partial_{t}D+\partial_{x}(D{v}) =\Gamma_{D}, \\
& \partial_{t}{v} + \partial_{x} \Bigg[ \dfrac{(3L-2)v^{2}}{2L} + \gamma (L + ln(1-L)) \Bigg] = \frac{\Gamma_{L} - M}{L(1- L)} {v}=\Gamma_{v}, \\
\end{aligned}
\right.
\end{equation}
where, from now on, $v:=v_{S}$. Obviously, we still have the biological constraints (\ref{satur}) and (\ref{conservazionemassatotale}), which close system (\ref{completo}).

In this paper, we aim to give a first analytical result to the Cauchy problem for this system  on the real line. The paper is organized as follows. In Section 2, we determine the region of hyperbolic symmetrizability of system (\ref{completo}), that ensures the local existence of smooth solutions (see \cite{Majda}, \cite{LiTaTsien}). Then, we establish some conditions in order to have a totally dissipative source. Finally, in the last section, we use this dissipative property to prove that, if we have initial data in a small neighborhood - in $H^2$-norm - of the unique equilibrium point $\bar{\textbf{u}} = (\bar{B}, \bar{E}, \bar{D}, \bar{v})$, these smooth solutions are global in time. In the last part of the paper, we analyze the asymptotic behavior of our solutions by using energy estimates. More precisely, we show that these solutions decay exponentially, in the $H^2$-norm, to the equilibrium point $\bar{\textbf{u}}$. \\ 
The proof of global existence of smooth solutions for system (\ref{completo}) takes inspiration from \cite{HN}, but here we do not make use of any strictly convex entropy. As we will discuss later, system (\ref{completo}) is ''totally dissipative'', which means that the source term has a particular dissipative property. Our proof of global existence for the solutions of (\ref{completo}) is based on this totally dissipative property, which means that dissipation prevents, for small initial data, the formation of shocks and ensures global existence of the solutions. 

\section{The quasilinear hyperbolic dissipative system}
Let us rewrite now system (\ref{completo}) in a compact form, in order to study its hyperbolic symmetrizability, which provides the local existence of smooth solutions (see \cite{Majda}, \cite{LiTaTsien}). Let us set
$$\textbf{u} = \left(\begin{array}{c}B \\E \\ D \\v\end{array}\right).$$
The flux function for system (\ref{completo}) is 
\begin{equation}
\label{flusso}
F(\textbf{u}) =  \left(\begin{array}{c}Bv \\Ev \\ Dv \\ \dfrac{(3L-2)v^{2}}{2L} + \gamma (L + ln(1-L))\end{array}\right),
\end{equation}
and 
\begin{equation}
\label{sorgenteG}
{G}(\textbf{u})=\left(\begin{array}{c}\Gamma_{B} \\\Gamma_{E} \\ \Gamma_{D} \\ \Gamma_{v}\end{array}\right)
\end{equation}
is the vector of the reaction terms. With this notation, we can write system (\ref{completo}) in the compact conservative form

$$\partial_{x}\textbf{u} + \partial_{x} F(\textbf{u}) = G(\textbf{u}).$$

To study the properties of (\ref{completo}), we need to compute the Jacobian matrix
\begin{equation}
\label{matriceA}
A(\textbf{u})= \left(\begin{array}{cccc}v & 0 & 0 & B \\0 & v & 0 & E \\0 & 0 & v & D \\\eta & \eta & \eta & \frac{(3L-2)}{L}
v\end{array}\right),
\end{equation}
 of the flux function $F(\textbf{u}).$ We can now write (\ref{completo}) as a quasilinear system of partial differential equation, in the following form

\begin{equation}
\label{compatta}
\partial_{t}\textbf{u} + A(\textbf{u}) \partial_{x}\textbf{u} = {G}(\textbf{u}).
\end{equation}

\subsection{Hyperbolicity and symmetrizability}
The main hypothesis of the theorem on local existence of smooth solutions for a general hyperbolic quasilinear system in form (\ref{compatta}) is symmetrizability (see \cite{Majda}, \cite{LiTaTsien}), namely the existence of a symmetric positive definite matrix $A_{0}$, such that $A_{0}A$ is also symmetric. 

Assume $L \neq 0$ and $L \neq 1$, and set
\begin{equation}
\label{etapositiva}
\eta := \frac{L\gamma}{(1-L)} - \frac{v^{2}}{L^{2}}.
\end{equation}

It easy to see that, if $\eta>0$, then a positive symmetrizer for the model in (\ref{completo}) is given by the matrix
\begin{equation}
\label{simmetrizzatore}
A_{0}(\textbf{u}) = \left(\begin{array}{cccc}\eta & 0 & 0 & 0 \\0 & \frac{B\eta}{E} & 0 & 0 \\0 & 0 & \frac{B\eta}{D} & 0 \\0 & 0 & 0 & B\end{array}\right).
\end{equation}

\begin{remark}
By the existence of a symmetrizer, system (\ref{compatta}), (\ref{matriceA}), (\ref{sorgenteG}) is symmetrizable hyperbolic in the following domain
\begin{equation}
\label{dominio}
 W = \Bigg\{ \textbf{u} = (B, E, D, v)\in [0,1]^3\times \R: \eta > 0, \ 0<L<1	\Bigg\}.
\end{equation}
Let us find now the domain of simple hyperbolicity of our model  (\ref{compatta}), (\ref{matriceA}), (\ref{sorgenteG}). We calculate the eigenvalues of $A(\textbf{u})$ in (\ref{matriceA}), which are the following:
\begin{enumerate}
\item $\lambda_{1, 2} = v;$
\item $\lambda_{2, 3} = \dfrac{(2L-1)v}{L} \pm \sqrt{(1-L)\Delta};$
\end{enumerate}
where 
\begin{equation}
\label{delta}
\Delta = \frac{L\gamma}{(1-L)} - \frac{v^{2}}{L} > 0,
\end{equation}
in order to have real eigenvalues. Moreover, observing (\ref{etapositiva}) we note that  $$\Delta > \eta > 0.$$ This inequality is completely in accordance with the general theory about hyperbolic symmetrizable system, because (strong) hyperbolicity follows directly from hyperbolic symmetrizability. 

Therefore, always if $\eta>0$, we can write the explicit expression of domain of symmetrizability (and then hyperbolicity) for system (\ref{compatta}), (\ref{matriceA}), (\ref{sorgenteG}), which is
\begin{equation}
\label{dominio}
 W = \Bigg\{ \textbf{u} = (B, E, D, v)\in [0,1]^3\times \R: -\frac{L^{3/2}\gamma^{1/2}}{(1-L)^{1/2}} < v < \frac{L^{3/2}\gamma^{1/2}}{(1-L)^{1/2}}, \ 0<L<1	\Bigg\}.
\end{equation}

Finally, let $\Omega$ a convex open subset of $W$. This set $\Omega$ will be our domain of work for system (\ref{compatta}), (\ref{matriceA}), (\ref{sorgenteG}).
\end{remark}
Let us set now $A_{1}(\textbf{u}) = A_{0}A(\textbf{u})$. We can write system (\ref{compatta}), (\ref{matriceA}), (\ref{sorgenteG}) in the following symmetric form
\begin{equation}
\label{compattasimmetrizzata}
A_{0}(\textbf{u}) \partial_{t} \textbf{u} + A_{1}(\textbf{u}) \partial_{x} \textbf{u} = A_{0}G(\textbf{u}),
\end{equation}
with $A_{0}, A, G$ given respectively in (\ref{simmetrizzatore}),  (\ref{matriceA}),  (\ref{sorgenteG}). \\
From symmetrizability, we have local in time smooth solution in $H^{2}(W)$ for system (\ref{compatta}), (\ref{matriceA}), (\ref{sorgenteG}) with initial data in the same space, thanks to the theorem of local well-posedness of the Cauchy problem for quasilinear hyperbolic symmetrizable systems (\cite{Majda}).






\subsection{Total dissipativity of the source}
Here we are going to show that, under some assumptions, the source term is totally dissipative and, then, under some smallness hypothesis on the initial data $\textbf{u}_{0}$, these solutions are also global in time. This dissipative property, in fact, allows to estimate the $H^{2}$-norm of solution $\textbf{u}$ and to prove that this norm is bounded for all $t > 0$, without the use of strictly convex and strictly dissipative entropy, as done in \cite{HN}). 

Let us give now our definition of \emph{Hyperbolic Symmetrizable Totally Dissipative System}. In order to clearly write this definition, we consider a general one-dimensional $n \times n$ hyperbolic symmetrizable system in the compact form (\ref{compatta}), with $\textbf{u} \in \Omega \subseteq \mathbb{R}^{n}$, where $\Omega$ is a convex open set and $A(\textbf{u}), {G}(\textbf{u})$ are smooth enough in $\Omega$. 

\begin{definition} [\textbf{(D)-Condition}]
\label{totaldissipative} Assume that system (\ref{compattasimmetrizzata}) is hyperbolic symmetrizable in a set $W$. Let $\Omega$ be a convex subset of $W$, which contains a unique equilibrium point $\bar{\textbf{u}}$ for (\ref{compattasimmetrizzata}), i.e. $ G(\bar{\textbf{u}})=0$. We say that system (\ref{compattasimmetrizzata}) is totally dissipative in $\Omega$, if   there exists a matrix $D=D(\textbf{u}, \bar{\textbf{u}}) \in {M}^{n \times n}$ such that, for every $\textbf{u}\in \Omega$, we have 

\begin{itemize}
\item ${G}(\textbf{u})=D(\textbf{u}, \bar{\textbf{u}})(\textbf{u}-\bar{\textbf{u}})$;

\item $A_{0}D(\textbf{u}, \bar{\textbf{u}})$ is negative definite.
\end{itemize}
\end{definition}
Let us check the totally dissipative property of system (\ref{completo}).
First of all, we determine the coordinates of the point in which the reaction terms vanish. Actually, we have only one isolated equilibrium point $\bar{\textbf{u}}$, which is inside of the domain of hyperbolicity. \\ 
To see that, set 
$$\bar{B} = \Bigg(1 - \dfrac{k_{D}}{k_{B}}\Bigg) \Bigg/ \Bigg(1+\dfrac{\alpha k_{D}}{k_{N}} + \dfrac{k_{D}k_{E}}{\epsilon k_{B}} \Bigg).$$ 
Therefore
\begin{equation}
\label{balancepoint}
\bar{\textbf{u}} = \bar{B} \left(\begin{array}{c}1 \\ \dfrac{k_{E} k_{D}}{\epsilon k_{B}} \\ \dfrac{\alpha k_{D}}{k_{N}} \\0\end{array}\right) =  \frac{a}{1+\frac{\alpha k_{D}}{k_{N}} + \frac{k_{D} k_{E}}{\epsilon k_{B}}} \left(\begin{array}{c} 1 \\ \dfrac{k_{E} k_{D}}{\epsilon k_{B}} \\ \dfrac{\alpha k_{D}}{k_{N}}  \\0\end{array}\right),
\end{equation}
with $a := \frac{k_{B} - k_{D}}{k_{B}}.$ It has been assumed that $\bar{B}, \bar{E}, \bar{D}, \bar{L} \in (0, 1)$. Then, from (\ref{balancepoint}), in order to have positive volume fractions $\bar{B}, \bar{E}, \bar{D}$, we have to assume
\begin{equation}
\label{puntodiequilibriominore1}
k_{B} > k_{D}.
\end{equation}
Let us observe now that inequality (\ref{dominio}), which describes the region of hyperbolic symmetrizability of (\ref{completo}), is verified near the equilibrium point $\bar{\textbf{u}}$. Then, if we take $\textbf{u}$ in a small neighborhood of this point, we are within the domain of hyperbolic symmetrizability. For this reason, to show our results of global existence of smooth solutions, the idea is to consider an initial data $\textbf{u}_{0}$ in a convex and compact neighborhood of this equilibrium point $\bar{\textbf{u}}$, then we set

\begin{equation}
\label{dominio1}
\Omega =\bar{\mathrm{B}}_{r}(\bar{\textbf{u}}), 
\end{equation}
for a fixed  $r > 0$. Let us write the matrix $D$, as in \textbf{(D)-Condition}, for the source term of system (\ref{completo}).  Then
$$G(\textbf{u}) =  D(\textbf{u},\bar{\textbf{u}})(\textbf{u} - \bar{\textbf{u}}),$$
with 
\begin{equation}
\label{matriceD}
D(\textbf{u},\bar{\textbf{u}}) = \left(\begin{array}{cccc}-Bk_{B} & -Bk_{B} & -Bk_{B} & 0 \\\frac{\epsilon \bar{E}(L-\bar{B})}{\bar{B}\bar{L}} & \frac{-\epsilon (BL+\bar{B}E)}{\bar{B}\bar{L}} & -\frac{\epsilon E}{\bar{L}} & 0 \\\frac{k_{N}D}{\bar{B}} & 0 & -\frac{k_{N}B}{\bar{B}} & 0 \\ 0 & 0 & 0 & \frac{\Gamma_{L}-M}{L(1-L)}\end{array}\right).
\end{equation}
In order to show that system in (\ref{compattasimmetrizzata}), (\ref{simmetrizzatore}), (\ref{matriceA}), (\ref{sorgenteG}) verifies the $\textbf{(D)-Condition}$ in a small neighborhood of $\bar{\textbf{u}}$, we have only to prove that matrix $A_{0}D(\bar{\textbf{u}},\bar{\textbf{u}})$ is negative definite. Using the expression of $\bar{\textbf{u}}$, we rewrite the expressions of the matrices
$$A_{0}(\bar{\textbf{u}}) = \left(\begin{array}{cccc} \frac{k_{D} \gamma}{(k_{B} - k_{D})}  & 0 & 0 & 0 \\ 0 & \frac{\epsilon k_{B} \gamma}{k_{E}(k_{B} - k_{D})}  & 0 & 0 \\ 0 & 0 & \frac{k_{N} \gamma}{\alpha (k_{B} -k_{D})} & 0 \\ 0 & 0 & 0 & \bar{B}\end{array}\right),$$
and 
$$D(\bar{\textbf{u}},\bar{\textbf{u}}) = \left(\begin{array}{cccc}-\bar{B}k_{B} & -\bar{B}k_{B} & -\bar{B}k_{B} & 0 \\k_{E}(\bar{L} - \bar{B}) & - \epsilon - \bar{B} k_{E} & -\bar{B} k_{E} & 0 \\ \alpha k_{D} & 0 & -k_{N} & 0 \\ 0 & 0 & 0 & - \frac{k_{B}^{2}M}{k_{D}(k_{B} - k_{D})}\end{array}\right).$$

Using the \emph{Routh-Hurwitz conditions} (\cite{Murray}) on $\frac{(A_{0}D) + (A_{0}D)^{T}}{2}$ , we aim to establish the sign of  $A_{0}D(\bar{\textbf{u}}, \bar{\textbf{u}})$. We have
\begin{equation}\begin{array}{l}
\Bigg(\frac{(A_{0}D) + (A_{0}D)^{T}}{2}\Bigg)(\bar{\textbf{u}}, \bar{\textbf{u}})\\ \\

\label{A0Dsimmetrizzata} = \left(\begin{array}{cccc}-\bar{B}k_{B}k_{D} & \frac{k_{B}}{2}[\epsilon(\bar{L} - \bar{B}) - k_{D}\bar{B}]  & \frac{k_{D}}{2}[k_{N} - \bar{B}k_{B}] & 0 \\\frac{k_{B}}{2}[\epsilon(\bar{L} - \bar{B}) - k_{D}\bar{B}] & -\frac{\epsilon k_{B}(\epsilon + k_{E} \bar{B})}{k_{E}} & -\frac{\epsilon k_{B} \bar{B}}{2}  & 0 \\\frac{k_{D}}{2}[k_{N} - \bar{B}k_{B}] & -\frac{\epsilon k_{B} \bar{B}}{2}  & -\frac{k_{N}^{2}}{\alpha} & 0 \\ 0 & 0 & 0 & - \frac{k_{B}^{2}M\bar{B}}{k_{D}(k_{B} - k_{D})}\end{array}\right).\end{array}
\end{equation}
Obviously, we are interested to the internal $3\times3$ matrix which contains the main diagonal of the originale $4 \times 4$ matrix. Its characteristic polynomial is $P(\lambda)$, and its general form is
\begin{equation}
\label{polinomio}
P(\lambda) = \lambda^{3} + a_{1} \lambda^{2} + a_{2} \lambda + a_{3}.
\end{equation}
Here, the coefficients $a_{i}$ ($ i = 1, 2, 3 $) are all real. We require the conditions on the $a_{i}$, such that the zeros of $P(\lambda)$ have $Re\lambda < 0.$ The necessary and sufficient conditions for this to hold are the \emph{Routh-Hurwitz conditions} \cite{Murray}. For the cubic equation, $\lambda^{3} + a_{1} \lambda^{2} + a_{2} \lambda + a_{3} = 0,$ the \emph{Routh-Hurwitz conditions} for $Re\lambda < 0$ are given by the following $[RH]$ conditions:
\begin{equation}[RH]
\left\{
\label{RouthHurwitz}
\begin{aligned}
&a_{1} > 0;\\
&a_{3} > 0;\\
&a_{1} a_{2} - a_{3} > 0.
\end{aligned}
\right.
\end{equation}

In our case, the coefficients of (\ref{polinomio}) are the following:
\begin{equation}
\left\{
\label{coefficients}
\begin{aligned}
&a_{1} = \bar{B} k_{B} k_{D} + \frac{\epsilon^{2} k_{B}}{k_{E}} + \epsilon \bar{B} + \frac{k_{N}^{2}}{\alpha};\\ \\
&a_{2} =  \frac{\epsilon^{2} \bar{B} k_{B}^{2} k_{D}}{k_{E}} + \frac{\epsilon \bar{B} k_{B} k_{N}^{2}}{\alpha} + \frac{\epsilon k_{B} k_{D} k_{N}^{2}}{\alpha} + \frac{\bar{B} k_{B} k_{D}(\epsilon \bar{B} k_{B} + \epsilon k_{D} + k_{D} k_{N} + \epsilon^{2})}{2}\\
&- \frac{\bar{B}^{2} k_{B}^{2}(k_{D}^{2} + \epsilon^{2})}{2} - \frac{k_{D}^{2}(k_{N}^{2} + \epsilon^{2})}{4}; \\ \\
&a_{3} = \frac{\epsilon^{2} \bar{B} k_{B}^{2} k_{D}^{2} k_{N}}{2 k_{E}} + \frac{\epsilon^{2} \bar{B} k_{B}^{2} k_{D} k_{N}^{2}}{\alpha k_{E}} + \frac{\epsilon \bar{B} k_{B} k_{D} k_{N}^{2} (\bar{B}k_{B} + \epsilon + k_{D})}{2 \alpha} \\
&+ \frac{\epsilon \bar{B}k_{B} k_{D}^{2} k_{N} (\epsilon + \bar{B} k_{B})}{4} - \frac{\epsilon \bar{B}k_{B}k_{D}(\epsilon \bar{B}k_{B} k_{N} + \epsilon \bar{B} k_{B} k_{D} + k_{D} k_{N}^{2})}{4} \\
&- \frac{\epsilon^{2} k_{B} k_{D}^{2} (k_{N}^{2} + \bar{B}^{2}k_{B}^{2})}{4 k_{E}} 
- \frac{k_{N}^{2}(\epsilon^{2} k_{D}^{2} + \epsilon^{2} \bar{B}^{2} k_{B}^{2} + \bar{B}^{2} k_{B}^{2} k_{D}^{2})}{4 \alpha}.
\end{aligned}
\right.
\end{equation}

Therefore, we can state our total dissipation result.
\begin{proposition}
Assuming that (\ref{puntodiequilibriominore1}) holds. Then, if $[RH]$ condition is valid for $a_1, a_{2}, a_{3}$ in (\ref{coefficients}), system (\ref{compattasimmetrizzata}), (\ref{simmetrizzatore}), (\ref{matriceA}), (\ref{sorgenteG}) satisfies \textbf{(D)-Condition}, and so it is totally dissipative in a neighborhood of its equilibrium point.
\end{proposition} 
Now, the first condition of (\ref{RouthHurwitz}) for $a_{1}$ in (\ref{coefficients}) is always verified. Regarding the second two conditions, it can be verified that they hold for some $\epsilon, \alpha, k_{B}, k_{D}, $  $k_{E}, k_{N}$, in particular for the values of the parameters in the following table, which was used in \cite{cdnr}:
\begin{center}
\begin{table}[ht]\centering\caption{A list of (dimensional) parameters \cite{cdnr}.}
\begin{tabular}{|c|c|c|c|}
\hline 
Param. & Value & Unit of meas. & Indications \\ 
\hline 
$k_{B}$ & $8 \cdot 10^{-6}$ & 1/sec & Bact. growth rate \\ 
\hline 
$k_{E}$ & $12 \cdot 10^{-6}$ & 1/sec & EPS growth rate \\ 
\hline 
$k_{D}$ & $2 \cdot 10^{-7}$ & 1/sec & Bact. death rate \\ 
\hline 
$k_{N}$ & $1 \cdot 10^{-6}$ & 1/sec & Dead cells consumption \\ 
\hline 
$\epsilon$ & $1.25 \cdot 10^{-7}$ & 1/sec & EPS death rate \\ 
\hline 
$\alpha$ & 0.25 & dimensionless & coeff. liquid dead-cells \\ 
\hline 
\end{tabular} 
\end{table}
\end{center}

More generally, if we consider a costant $a \in [0.5, 1.5]$, we can restrict our attention at the class of coefficients such that 
\begin{equation}\label{param} \epsilon = 1.25 \cdot 10^{-7} \simeq 10^{-7}, ~~ k_{N} = 10 a \epsilon, ~~, k_{E} = 100 a \epsilon, ~~ k_{D} = 2 a \epsilon, ~~ k_{B} = 70 a \epsilon.
\end{equation}
By this reduction, condition for $a_{3}$ (the third condition of (\ref{RouthHurwitz})) becomes a second degree polynomial inequality
$$ - 3.366175 \cdot 10^{10} a^{2} + 4.1829869 \cdot 10^{10} a - 8.303370950 \cdot 10^{9} > 0,$$
that can be easily solved, and the inequality holds for $a \in [0.5, 1] $. Moreover, by a simple numerical verification, it can be seen that  the second condition in  (\ref{RouthHurwitz}) is also true in this interval. Then, if we take the coefficients as in \eqref{param}, then $\textbf{(D)-Condition}$ holds ture for $a \in [0.5, 1]$.

\section{Global existence of smooth solutions for small initial data and  asymptotic behavior}

Let us present now  the proof of the global existence of smooth solutions to system (\ref{completo}) with initial data in a small neighborhood, in $H^{2}$-norm, of the equilibrium point. We want to follow the proof of global existence of smooth solutions for weakly dissipative hyperbolic systems with a convex entropy in  \cite{HN} and \cite{BHN}. Using \textbf{(D)-Condition}, we can write the $L^2$-norm estimates of the local solution $\textbf{u}$ and of its first and second derivatives. Then, global existence of smooth solutions follows from these estimates. We underline that, in order to get these estimates, we use only the dissipative property of the source term - \textbf{(D)-Condition} - and we do not use any strictly convex and strictly dissipative entropy for system (\ref{completo}), which is an assumption of the theorem of global existence in \cite{HN}.  \\
Now, if we take the initial data in a small neighborhood of the equilibrium point $\bar{\textbf{u}}$, and we assume that our parameters are such that the  $\textbf{(D)-Condition}$ is verified,  we are able to prove global existence of smooth solutions to model (\ref{completo}). Moreover,  using the totally dissipative property of \textbf{(D)-Condition}, we can prove that this global solution will decay to the unique equilibrium point.
In the following, let us state our global existence and asymptotic behavior result. 
\begin{theorem}
\label{Theorem}
Consider system (\ref{completo}) and its unique equilibrium point $\bar{\textbf{u}}$ in (\ref{balancepoint}) and assume that (\ref{puntodiequilibriominore1}) holds. If this system satisfies \textbf{(D)-Condition}, then there exists a $0 < \delta < r$ such that, if $||\textbf{u}_{0}-\bar{\textbf{u}}||_{2} \le \delta $, then there is a unique global solution $\textbf{u}$ with initial data $\textbf{u}_{0}$, which verifies
$$\textbf{u}-\bar{\textbf{u}} \in C([0, +\infty), H^{2}(\mathbb{R})) \cap C^{1}([0, +\infty), H^{1}(\mathbb{R}))$$
and
\begin{equation}
\label{funzio}
sup_{~ 0 \le T < +\infty~}||(\textbf{u}-\bar{\textbf{u}})(t)||_{2}^{2} + \int_{0}^{T} ||(\textbf{u}-\bar{\textbf{u}})(t)||_{2}^{2} ~ d\tau \le C(\delta)||\textbf{u}_{0}-\bar{\textbf{u}}||_{2},
\end{equation}
where  $C(\delta)$ is a positive constant.
Moreover, the global solution $\textbf{u}$ decays exponentially in $H^{2}$-norm to the equilibrium point $\bar{\textbf{u}}$, i.e.
\begin{equation}
\label{asymptotic}
||\textbf{u} - \bar{\textbf{u}}||_{H^{2}(\mathbb{R})} \le C_{1} e^{~ \beta t} ||\textbf{u}_{0} - \bar{\textbf{u}}||_{H^{2}(\mathbb{R})},
\end{equation}
where $C_{1}, \beta$ are positive constants.
\end{theorem}

In accordance with the hypothesis, we work in a neighborhood $\Omega$ of the equilibrium point $\bar{\textbf{u}}$. Let us introduce the new unknown
\begin{equation}
\textbf{w}:=\textbf{u}-\bar{\textbf{u}}. 
\end{equation}
Then, the equation in (\ref{compattasimmetrizzata}) becomes
\begin{equation}
\label{compatta1}
A_{0}(\textbf{w}+\bar{\textbf{u}})\partial_{t}\textbf{w} + A_{1}(\textbf{w}+\bar{\textbf{u}})\partial_{x}\textbf{w}=A_{0}{G}(\textbf{w}+\bar{\textbf{u}}).
\end{equation}

To prove global existence for (\ref{completo}), we follow the  approach proposed in   \cite{Nishida}, see also \cite{Kawashima} and \cite{HN}, and we introduce the functional
\begin{equation}
\label{funzionale}
N_{l}^{2}(t) := sup_{~ 0 \le \tau \le t ~}||\textbf{w}(\tau)||_{l}^{2} + \int_{0}^{t} ||\textbf{w}(\tau)||_{l}^{2} ~ d\tau,
\end{equation}
for $l = 0,1,2.$

\begin{proposition}
\label{proposizione}
Let $T > 0$, and assume that there exists a smooth solution $\textbf{w}$ of (\ref{compattasimmetrizzata}), (\ref{simmetrizzatore}), (\ref{matriceA}), (\ref{sorgenteG}) in $[0, T]$. Then, there exists $\epsilon > 0$ and $C > 0$ such that, if $N_{2}(T) \le \epsilon$, 

\begin{equation}
\label{disuguaglianzafondamentale}
N_{2}^{2}(T) \le C(N_{2}^{2}(0) + N_{2}^{3}(T)).
\end{equation}
\end{proposition}
In a classical way, the first part of \emph{Theorem \ref{Theorem}} about global existence and uniqueness of smooth solutions follows directly from \emph{Proposition \ref{proposizione}} (see \cite{HN}, \cite{Nishida}).
The following Lemma is the first step of the proof of global existence. 

\begin{lemma}
\label{lemma1}
If $N_{2}(T) \le \epsilon$, then 
\begin{equation}
\label{lemma1eq}
N_{0}^{2}(T) \le C_{1}(N_{2}^{2}(0) + N_{2}^{3}(T)) .  
\end{equation}
\end{lemma}
Usually, to state an estimate as (\ref{lemma1eq}), a function of convex entropy is  used, but here, in our proof of Lemma \ref{lemma1}, we do not use anything but the dissipative property of system (\ref{completo}).

\begin{proof}
Using $\textbf{(D)-Condition}$ and (\ref{compatta1}), let us consider system (\ref{completo}) in the following symmetric form:
\begin{equation}
A_{0}(\textbf{w}+\bar{\textbf{u}})\partial_{t}{\textbf{w}} + A_{1}(\textbf{w}+\bar{\textbf{u}})\partial_{x}{\textbf{w}}= A_{0}(\textbf{w}+\bar{\textbf{u}})D(\textbf{u}, \bar{\textbf{u}}){\textbf{w}}. \\
\end{equation} 

In the previous equation, the new reaction term is
\begin{equation}
\label{sorgente}
D_{1}(\textbf{u}, \bar{\textbf{u}}, \textbf{w}) := A_{0}(\textbf{w}+\bar{\textbf{u}})D(\textbf{u}, \bar{\textbf{u}}).
\end{equation}

Therefore, we have
\begin{equation}
\label{compatta2}
A_{0}(\textbf{w}+\bar{\textbf{u}})\partial_{t}\textbf{w}+ A_{1}(\textbf{w} + \bar{\textbf{u}})\partial_{x}\textbf{w} = D_{1}(\textbf{u}, \bar{\textbf{u}}, \textbf{w}){\textbf{w}}. 
\end{equation} 

We have the following identities:
\begin{equation}
\label{id1}
( A_{0}(\textbf{w}+\bar{\textbf{u}})\partial_{t}{\textbf{w}}, {\textbf{w}} ) = \frac{1}{2} \partial_{t}(A_{0}(\textbf{w}+\bar{\textbf{u}}){\textbf{w}},{\textbf{w}} ) - \frac{1}{2} (\partial_{t}A_{0}(\textbf{w}+\bar{\textbf{u}}){\textbf{w}}, {\textbf{w}} );
\end{equation} 

\begin{equation}
\label{id2}
( A_{1}(\textbf{w}+\bar{\textbf{u}})\partial_{x}{\textbf{w}},{\textbf{w}} ) = \frac{1}{2} \partial_{x}(A_{1}(\textbf{w}+\bar{\textbf{u}}){\textbf{w}}, {\textbf{w}} ) - \frac{1}{2} (\partial_{x}A_{1}(\textbf{w}+\bar{\textbf{u}})\bar{\textbf{w}}, {\textbf{w}} ).
\end{equation} 

We consider (\ref{compatta2}) and take the inner product with ${\textbf{w}}$, which yields
\begin{equation}
\label{scalare}
(A_{0}(\textbf{w}+\bar{\textbf{u}}) \partial_{t}{\textbf{w}},{\textbf{w}}) +( A_{1}(\textbf{w}+\bar{\textbf{u}})\partial_{x}{\textbf{w}},{\textbf{w}}) = (D_{1}(\textbf{u}, \bar{\textbf{u}}, \textbf{w}){\textbf{w}},{\textbf{w}}). 
\end{equation} 

Using the identities (\ref{id1}) and (\ref{id2}) in (\ref{scalare}), we obtain
\begin{equation*}
\frac{1}{2}\partial_{t}(A_{0}(\textbf{w}+\bar{\textbf{u}}) {\textbf{w}}, {\textbf{w}}) + \frac{1}{2}\partial_{x}( A_{1}(\textbf{w}+\bar{\textbf{u}})){\textbf{w}},{\textbf{w}}) 
\end{equation*}

\begin{equation}
\label{scalare1}
=  \frac{1}{2}(\partial_{t}A_{0}(\textbf{w}+\bar{\textbf{u}}) {\textbf{w}},{\textbf{w}}) +\frac{1}{2}(\partial_{x}A_{1}(\textbf{w}+\bar{\textbf{u}}) {\textbf{w}},{\textbf{w}}) + (D_{1}(\textbf{u}, \bar{\textbf{u}}, \textbf{w}){\textbf{w}},{\textbf{w}}). 
\end{equation}

Therefore, if we integrate equality (\ref{scalare}) over $\mathbb{R}\times[0, T]$, we have

\begin{equation*}
\frac{1}{2}\int_{\mathbb{R}} (A_{0}(\textbf{w}(T)+\bar{\textbf{u}}){\textbf{w}}(T),{\textbf{w}}(T))  ~ dx ~  - \frac{1}{2}\int_{\mathbb{R}} (A_{0}(\textbf{w}(0)+\bar{\textbf{u}}){\textbf{w}}(0),{\textbf{w}}(0)) ~ dx ~ 
\end{equation*} 

\begin{equation}
= \int_{0}^{T}  dt ~\int_{\mathbb{R}}\Bigg ( \left[\frac{1}{2}\partial_{t}A_{0}(\textbf{w}+\bar{\textbf{u}}) + \frac{1}{2}\partial_{x}A_{1}(\textbf{w}+\bar{\textbf{u}}) +  D_{1}(\textbf{u}, \bar{\textbf{u}}, \textbf{w})\right]{\textbf{w}},{\textbf{w}} \Bigg ) ~ dx =: \mathrm{I}.
\end{equation} 

To estimate $\mathrm{I}$, we use (\ref{compatta2}) in the following form:

\begin{equation}
\label{derivata1}
\partial_{t}{\textbf{w}} =  -A (\textbf{w}+\bar{\textbf{u}}) \partial_{x}{\textbf{w}} + D(\textbf{u}, \bar{\textbf{u}}, \textbf{w}){\textbf{w}}. 
\end{equation} 

Then
\begin{equation}
\label{derivata2}
\partial_{t}A_{0} (\textbf{w}+\bar{\textbf{u}})= A'_{0} (\textbf{w}+\bar{\textbf{u}}) \partial_{t}{\textbf{w}},
\end{equation} 

and
\begin{equation}
\partial_{t}A_{0} (\textbf{w}+\bar{\textbf{u}}) =  A'_{0} (\textbf{w}+\bar{\textbf{u}})(- A (\textbf{w}+\bar{\textbf{u}}) \partial_{x}{\textbf{w}} + D(\textbf{u}, \bar{\textbf{u}}, \textbf{w}){\textbf{w}}). 
\end{equation} 

Therefore, we have
\begin{equation*}
\mathrm{I} = - \int_{0}^{T}  dt ~\int_{\mathbb{R}} \frac{1}{2}((A'_{0}A(\textbf{w}+\bar{\textbf{u}})\partial_{x}{\textbf{w}},{\textbf{w}}),{\textbf{w}}) ~ dx ~  
\end{equation*} 

$$+ \int_{0}^{T}  dt ~\int_{\mathbb{R}]}\frac{1}{2}((A'_{1}(\textbf{w}+\bar{\textbf{u}})\partial_{x}{\textbf{w}},{\textbf{w}}),{\textbf{w}}) ~ dx$$

\begin{equation}
+ \int_{0}^{T}  dt ~ \int_{\mathbb{R}} D_{1}(\textbf{u}, \bar{\textbf{u}}, \textbf{w}) ~ dx +  \int_{0}^{T}  dt ~ \int_{\mathbb{R}} \frac{1}{2}(A'_{0}D(\textbf{u}, \bar{\textbf{u}}) \bar{\textbf{w}} \cdot {\textbf{w}},{\textbf{w}}) ~ dx.
\end{equation} 

Then
\begin{equation*}
\frac{1}{2}\int_{\mathbb{R}} (A_{0}(\textbf{w}(T)+\bar{\textbf{u}}){\textbf{w}}(T),{\textbf{w}}(T)) ~ dx ~ - \int_{0}^{T}  dt ~ \int_{\mathbb{R}} (D_{1}(\textbf{u}, \bar{\textbf{u}}, \textbf{w}) ~ dx ~ 
\end{equation*} 

\begin{equation*}
= \frac{1}{2}\int_{\mathbb{R}} (A_{0}(\textbf{w}(0)+\bar{\textbf{u}}){\textbf{w}}(0),{\textbf{w}}(0)) ~ dx ~ +  \int_{0}^{T}  dt ~\int_{\mathbb{R}}\frac{1}{2}((A'_{1}(\textbf{w}+\bar{\textbf{u}})\partial_{x}{\textbf{w}},{\textbf{w}}),{\textbf{w}}) ~ dx  
\end{equation*}

\begin{equation}
\label{quasiprimastima}
+  \int_{0}^{T}  dt ~ \int_{\mathbb{R}} \frac{1}{2} ((A'_{0}D(\textbf{u}, \bar{\textbf{u}}){\textbf{w}}, {\textbf{w}}),{\textbf{w}}) -  \int_{0}^{T}  dt ~\int_{\mathbb{R}} \frac{1}{2}((A'_{0}A(\textbf{w}+\bar{\textbf{u}})\partial_{x}{\textbf{w}}),{\textbf{w}},{\textbf{w}}) ~ dx .
\end{equation}
 
Since $A_{0}$ is positive definite, $D_{1}$ is negative definite and since $A_{0}, A'_{1}, A'_{0}D, A'_{0}A$ are bounded in a neighborhood of the equilibrium point $\bar{\textbf{u}},$ we have
\begin{equation*}
\frac{c}{2} ||{\textbf{w}}(T)||_{0}^{2} + {c_{1}} \int_{0}^{T} ||{\textbf{w}}(t)||_{0}^{2} ~ dt 
\end{equation*}

\begin{equation*}
\le \frac{c_{2}}{2}||{\textbf{w}}(0)||_{0}^{2} + \frac{c_{3}}{2}\int_{0}^{T}  dt ~\int_{\mathbb{R}}|\partial_{x}{\textbf{w}}| |\textbf{w}|^{2} ~ dx + \frac{c_{4}}{2}  \int_{0}^{T}  dt ~\int_{\mathbb{R}} |\partial_{x}{\textbf{w}}| |{\textbf{w}}|^{2} ~ dx 
\end{equation*}

\begin{equation*}
+ \frac{c_{5}}{2} \int_{0}^{T}  dt ~ \int_{\mathbb{R}} |{\textbf{w}}| |{\textbf{w}}|^{2} ~ dx
\end{equation*}

\begin{equation}
\label{primastima}
\le \frac{c_{2}}{2}||{\textbf{w}}(0)||_{0}^{2} + \frac{c_{6}}{2} ||\partial_{x}{\textbf{w}}||_{\infty} \int_{0}^{T}  dt ~ ||{\textbf{w}}||_{0}^{2} + \frac{c_{5}}{2} ||{\textbf{w}}||_{\infty} \int_{0}^{T}  dt ~ ||{\textbf{w}}||_{0}^{2},
\end{equation}

where $c, c_{1}, c_{2}, c_{3}, c_{4}, c_{5}, c_{6} \in \mathbb{R^{+}}$. \\

The embedding of  $H^{1}$ in $L_{\infty}$, where $\alpha$ is the constant of the embedding, yields

\begin{equation}
||\textbf{w}||_{\infty} \le \alpha ||\textbf{w}||_{H^{1}} = \alpha (||\textbf{w}||_{0} + ||\partial_{x} \textbf{w}||_{0}).
\end{equation}

Thus, from the definition of functional $N_{2}(t)$ in (\ref{funzionale}), we have
\begin{equation}
||\partial_{x}{\textbf{w}}||_{\infty} \le \alpha (||\partial_{x}{\textbf{w}}||_{0} + ||\partial_{xx}{\textbf{w}}||_{0}) \le N_{2}(T),
\end{equation}

and 
\begin{equation}
||{\textbf{w}}||_{\infty} \le \alpha (||{\textbf{w}}||_{0} + ||\partial_{x}{\textbf{w}}||_{0}) \le N_{2}(T).
\end{equation}

The last term in (\ref{primastima}) is estimated by
\begin{equation}
\frac{c_{2}}{2} N_{2}^{2}(0) + \frac{c_{6}}{2} N_{2}^{3}(T) + \frac{c_{5}}{2} N_{2}^{3}(T).
\end{equation}
So, using (\ref{ultima}), we have

\begin{equation}
\frac{c}{2} ||{\textbf{w}}(T)||_{0}^{2} + \frac{c_{1}}{2} \int_{0}^{T} ||{\textbf{w}}(t)||_{0}^{2} ~ dt \le C(N_{2}^{2}(0) + N_{2}^{3}(T)),
\end{equation}
and, therefore
 
\begin{equation}
||{\textbf{w}}(T)||_{0}^{2} + \int_{0}^{T} ||{\textbf{w}}(t)||_{0}^{2} ~ dt \le C_{1}(N_{2}^{2}(0) + N_{2}^{3}(T)).
\end{equation}
\end{proof}

Let us now estimate the first and second order derivatives.

\begin{lemma} 
\label{lemma2}
If $N_{2}(T) \le \epsilon $, Then, for $l = 1,2,$
\begin{equation}
sup_{~ 0\le t \le T~} ||\textbf{w}(t)||_{l}^{2}+\int_{0}^{T} ||\textbf{w}(t)||_{l}^{2} ~ dt \le C(N_{2}^{2}(0)+N_{2}^{3}(T)).
\end{equation}
\end{lemma}

\begin{proof}
Apply the first space derivative to system (\ref{compatta1}) and take the inner product with $\partial_{x}\textbf{w}$, which provides
\begin{equation}
\label{D}
\partial_{x}(A_{0}(\textbf{w}+\bar{\textbf{u}})\partial_{t}\textbf{w} + A_{1}(\textbf{w}+\bar{\textbf{u}})\partial_{x}\textbf{w}) \cdot \partial_{x}\textbf{w} =[\partial_{x}(A_{0}D)\textbf{w} + (A_{0}D)\partial_{x}\textbf{w}] \cdot \partial_{x}\textbf{w}.
\end{equation}

We have the following identities:
\begin{equation*}
\partial_{x}(A_{0}(\textbf{w}+\bar{\textbf{u}})\partial_{t}\textbf{w}) \cdot \partial_{x}\textbf{w} = \frac{1}{2}\partial_{t}((A_{0}(\textbf{w}+\bar{\textbf{u}})\partial_{x}\textbf{w}) \cdot \partial_{x}\textbf{w}) - \frac{1}{2}(\partial_{t}A_{0}\partial_{x}\textbf{w}) \cdot \partial_{x}\textbf{w} 
\end{equation*}

\begin{equation}
+(\partial_{x}A_{0}\partial_{t}\textbf{w}) \cdot \partial_{x} \textbf{w};
\end{equation}

\begin{equation}
\partial_{x}(A_{1}(\textbf{w}+\bar{\textbf{u}})\partial_{x}\textbf{w}) \cdot \partial_{x} \textbf{w} = \frac{1}{2}\partial_{x}((A_{1}(\textbf{w}+\bar{\textbf{u}})\partial_{x}\textbf{w}) \cdot \partial_{x}\textbf{w}) + \frac{1}{2}(\partial_{x}A_{1}\partial_{x}\textbf{w}) \cdot \partial_{x} \textbf{w}.
\end{equation}

If we integrate equality (\ref{D}) over $\mathbb{R}$ and use the previous identities, the term
$$\partial_{x}((A_{1}(\textbf{w}+\bar{\textbf{u}})\partial_{x}\textbf{w}) \cdot \partial_{x}\textbf{w})$$ vanishes and then we have

\begin{equation*}
\frac{1}{2}\frac{d}{dt} \int_{\mathbb{R}} (A_{0}(\textbf{w}+\bar{\textbf{u}})\partial_{x}\textbf{w}) \cdot \partial_{x}\textbf{w} ~ dx - \int_{\mathbb{R}} ((A_{0}D)\partial_{x}\textbf{w}) \cdot \partial_{x}\textbf{w} ~ dx
\end{equation*}

\begin{equation}
\label{integrata}
= \int_{\mathbb{R}} \Bigg\{ \frac{1}{2}\partial_{t}A_{0}\partial_{x}\textbf{w} - \partial_{x}A_{0}\partial_{t}\textbf{w}- \frac{1}{2}\partial_{x}A_{1}\partial_{x}\textbf{w} \Bigg\} \cdot \partial_{x}\textbf{w} + (\partial_{x}(A_{0}D)\textbf{w}) \cdot \partial_{x}\textbf{w} ~ dx.
\end{equation}

To estimate the right-end side of (\ref{integrata}), we use (\ref{compatta1}) in the following form:
\begin{equation}
\partial_{t}\textbf{w} = - A\partial_{x}\textbf{w}+D\textbf{w}.
\end{equation}

Then

\begin{equation*}
\partial_{t}A_{0} = A_{0}'\partial_{t}\textbf{w} = A_{0}'(- A\partial_{x}\textbf{w}+D\textbf{w}).
\end{equation*}

Thus, equality (\ref{integrata}) is

\begin{equation*}
\frac{1}{2}\frac{d}{dt} \int_{\mathbb{R}} (A_{0}(\textbf{w}+\bar{\textbf{u}})\partial_{x}\textbf{w}) \cdot \partial_{x}\textbf{w} ~ dx - \int_{\mathbb{R}} ((A_{0}D)\partial_{x}\textbf{w}) \cdot \partial_{x}\textbf{w} ~ dx
\end{equation*}

$$=\int_{\mathbb{R}} -\frac{1}{2}((A_{0}'A\partial_{x}\textbf{w}, \partial_{x}\textbf{w}), \partial_{x}\textbf{w}) + \frac{1}{2} (A_{0}'\partial_{x}\textbf{w}A\partial_{x}\textbf{w}, \partial_{x}\textbf{w}) + \frac{1}{2} (A_{0}'D\textbf{w} \partial_{x}\textbf{w}, \partial_{x}\textbf{w})~ dx$$

\begin{equation}
\label{ultima}
-\frac{1}{2}((A_{0}A'\partial_{x}\textbf{w}, \partial_{x}\textbf{w}), \partial_{x}\textbf{w})+\int_{\mathbb{R}} (A_{0}D'(\partial_{x}\textbf{w})\textbf{w}, \partial_{x}\textbf{w}) ~ dx.
\end{equation}

Using (\ref{ultima}), we have
$$||\partial_{x}\textbf{w}(T)||_{0}^{2} + \int_{0}^{T} ||\partial_{x}\textbf{w}(t)||_{0}^{2} ~ dt$$

$$\le C_{1}(\epsilon) \Bigg( ||\partial_{x}\textbf{w}(0)||_{0}^{2} + \int_{0}^{T} \int_{\mathbb{R}} |\partial_{x}\textbf{w}|^{2}|\partial_{x}\textbf{w}| ~ dxdt \Bigg)$$

\begin{equation}
\label{stimaderivataprima}
\le C_{1}(\epsilon) \Bigg( ||\partial_{x}\textbf{w}(0)||_{0}^{2} + (||\textbf{w}||_{\infty}+||\partial_{x}\textbf{w}||_{\infty})\int_{0}^{T} ||\partial_{x}\textbf{w}(t)||_{0}^{2} ~ dt \Bigg).
\end{equation}

\vspace{5mm}
In the same way, we perform the second space derivative of (\ref{compatta2}) and take the inner product with $\partial_{xx}\textbf{w}$, which provides

\begin{equation}
\label{duederivate}
\partial_{xx}(A_{0}\partial_{t}\textbf{w} + A_{1}\partial_{x}\textbf{w}) \cdot \partial_{xx}\textbf{w} = \partial_{xx}(A_{0}D\textbf{w}) \cdot \partial_{xx}\textbf{w}.
\end{equation}

We have the following identities:
\begin{equation}
\partial_{xx}(A_{0}\partial_{t}\textbf{w}) \cdot \partial_{xx}\textbf{w} 
\end{equation}

\begin{equation}
= \frac{1}{2}\partial_{t}((A_{0}\partial_{xx}\textbf{w}) \cdot \partial_{xx}\textbf{w}) - \frac{1}{2}(\partial_{t}A_{0}\partial_{xx}\textbf{w}) \cdot \partial_{xx}\textbf{w}
\end{equation}

\begin{equation}
\label{iden1}
+ 2 (\partial_{x}A_{0}\partial_{xt}\textbf{w}) \cdot \partial_{xx}\textbf{w} + (\partial_{xx}A_{0}\partial_{t}\textbf{w}) \cdot \partial_{xx}\textbf{w};
\end{equation}

\vspace{5mm}

\begin{equation*}
\partial_{xx}(A_{1}(\textbf{W})\partial_{x}\textbf{w}) \cdot \partial_{xx}\textbf{w} 
\end{equation*}

\begin{equation}
\label{iden2}
= \frac{1}{2}\partial_{x}((A_{1}\partial_{xx}\textbf{w}) \cdot \partial_{xx}\textbf{w}) + \frac{3}{2}(\partial_{x}A_{1}\partial_{xx}\textbf{w}) \cdot \partial_{xx}\textbf{w} + (\partial_{xx}A_{1}\partial_{x}\textbf{w}) \cdot \partial_{xx}\textbf{w}.
\end{equation}

If we integrate (\ref{duederivate}) over $\mathbb{R}$ and we use the previous identities (\ref{iden1})-(\ref{iden2}), the term
$$\partial_{x}((A_{1}\partial_{xx}\textbf{w}) \cdot \partial_{xx}\textbf{w})$$

vanishes.\\

In this way, we have
\begin{equation*}
\frac{1}{2}\frac{d}{dt} \int_{\mathbb{R}} (A_{0}\partial_{xx}\textbf{w}) \cdot \partial_{xx}\textbf{w} ~ dx - \int_{\mathbb{R}} (A_{0}D\partial_{xx}\textbf{w}, \partial_{xx}\textbf{w}) ~ dx 
\end{equation*}

\begin{equation*}
= \int_{\mathbb{R}} \Bigg\{\frac{1}{2}\partial_{t}A_{0}\partial_{xx}\textbf{w} -2\partial_{x}A_{0}\partial_{xt}\textbf{w} - \partial_{xx}A_{0}\partial_{t}\textbf{w}\Bigg\} \cdot \partial_{xx}\textbf{w} ~ dx
\end{equation*}

\begin{equation*}
- \int_{\mathbb{R}} \Bigg \{\frac{3}{2}\partial_{x}A_{1}\partial_{xx}\textbf{w} + \partial_{xx}A_{1}\partial_{x}\textbf{w}\Bigg\} \cdot \partial_{xx}\textbf{w} + (\partial_{xx}(A_{0}D) \textbf{w}, \partial_{xx}\textbf{w}) + ((A_{0}D) \partial_{x}\textbf{w},  \partial_{xx}\textbf{w})~ dx
\end{equation*}

$$= \int_{\mathbb{R}} \Bigg\{ -\frac{1}{2} (A_{0}'A \partial_{x}\textbf{w}, \partial_{xx}\textbf{w}) + \frac{1}{2} (A_{0}'D \textbf{w}, \partial_{xx}\textbf{w})  + (A_{0}'  \partial_{x}\textbf{w}A' \partial_{x}\textbf{w}, \partial_{x}\textbf{w})$$

$$+\frac{1}{2}A_{0}'\partial_{x}\textbf{w}A\partial_{xx}\textbf{w} - 2(A_{0}'\partial_{x}\textbf{w}D'\partial_{x}\textbf{w}, \textbf{w}) - 2A_{0}'\partial_{x}\textbf{w}D\partial_{x}\textbf{w}-A_{0}'\partial_{xx}\textbf{w}D\textbf{w} $$

\begin{equation}
\label{perstima}
- (A_{0}A''\partial_{x}\textbf{w} \cdot \partial_{x}\textbf{w}, \partial_{x}\textbf{w}) - (A_{0}A' \partial_{xx}\textbf{w}, \partial_{x}\textbf{w}) + (A_{0}'D'\partial_{xx}\textbf{w},\textbf{w}) \Bigg\} \cdot \partial_{xx}\textbf{w} ~ dx.
\end{equation}

Therefore
\begin{equation*}
||\partial_{xx}\textbf{w}(T)||_{0}^{2} + \int_{0}^{T} ||\partial_{xx}\textbf{w}(t)||_{0}^{2} ~  dt
\end{equation*}

\begin{equation*}
\le C_{2}(\epsilon) \Bigg\{||\partial_{xx}\textbf{w}(0)||_{0}^{2} + \int_{0}^{T} \int_{[0,1]} (|\partial_{x}\textbf{w}|^{2} + |\partial_{xx}\textbf{w}|^{2}) (|\partial_{x}\textbf{w}| + |\textbf{w}| + |\partial_{x}\textbf{w}|) ~  dx dt \Bigg \}
\end{equation*}

\begin{equation}
\label{stimaderivataseconda}
\le C_{2}(\epsilon) \Bigg \{||\partial_{xx}\textbf{w}(0)||_{0}^{2} + (||\textbf{w}||_{\infty} + ||\partial_{x}\textbf{w}||_{\infty}) \int_{0}^{T} (||\partial_{x}\textbf{w}(t)||_{0}^{2} + ||\partial_{xx}\textbf{w}(t)||_{0}^{2}) ~ dt \Bigg \}.
\end{equation}
\end{proof}

We note that, by \emph{Lemma \ref{lemma1}} and \emph{Lemma \ref{lemma2}}, we are able to prove inequality (\ref{disuguaglianzafondamentale}). Then, \emph{Proposition \ref{proposizione}} is proved. \\
To conclude the proof of \emph{Theorem \ref{Theorem}}, we are going to prove the  statement about the asymptotic behavior of solutions. \\ 

Let us set 

\begin{equation}
\label{energy}
\mathrm{E}(t) = \dfrac{1}{2} ||\textbf{w}(t)||^{2}_{2}.
\end{equation}

Taking the time derivative of (\ref{quasiprimastima}) and using the embedding of $H^{2}(\mathbb{R})$ in $L^{\infty}(\mathbb{R})$, we have 

\begin{equation}
\label{zerorder}
\frac{1}{2}\frac{d}{dt}||\textbf{w}(t)||_{0}^{2} + c_{1}||\textbf{w}(t)||_{0}^{2} \le c_{2} ||\textbf{w}(t)||_{2} ||\textbf{w}(t)||_{0}^{2}.
\end{equation}

Taking the time derivative of (\ref{ultima}), we obtain, in the same way, the estimate on the first derivative of $\textbf{w}$. Therefore

\begin{equation}
\label{firstorder}
\frac{1}{2}\frac{d}{dt}||\partial_{x}\textbf{w}(t)||_{0}^{2} + c_{3}||\partial_{x}\textbf{w}(t)||_{0}^{2} \le c_{4} ||\textbf{w}(t)||_{2} ||\partial_{x}\textbf{w}(t)||_{0}^{2}.
\end{equation}

Finally, from the time derivative of (\ref{perstima}) and using Morrey's Theorem, we have the second order estimate

\begin{equation}
\label{secondorder}
\frac{1}{2}\frac{d}{dt}||\partial_{xx}\textbf{w}(t)||_{0}^{2} + c_{5}||\partial_{xx}\textbf{w}(t)||_{0}^{2} \le c_{6}||\textbf{w}(t)||_{2} ||\partial_{xx}\textbf{w}(t)||_{0}^{2}.
\end{equation}

We note that $c_{1}, c_{3}, c_{5}$ depend on the norm of $A_{0}, D$, while $c_{2}, c_{4}, c_{6}$ depend on the norm of $A_{0}, A_{0}', A, A', D, D'.$
Summing (\ref{zerorder}), (\ref{firstorder}) and (\ref{secondorder}), from (\ref{energy}) we have
\begin{equation}
\label{Gronwall}
\partial_{t}\mathrm{E} + \mu \mathrm{E} \le \nu \mathrm{E}^{3/2},
\end{equation}

where $\mu=c_{1}+c_{3}+c_{5}$ and $\nu=c_{2}+c_{4}+c_{6}$. \\

Now, thanks to the assumptions, we can choose the initial data small enough such that
 $$\mathrm{E}(0) < \Bigg({\frac{\gamma}{\nu}}\Bigg)^{2},$$ 
for some $0 < \gamma < min\{\mu, \nu\} $.
From the smallness of the initial data, we can obtain the inequality
$$\nu \mathrm{E}^{3/2} < \gamma \mathrm{E},$$
at least for small times. Therefore, using (\ref{Gronwall}) we have
$$\partial_{t}\mathrm{E} \le (\gamma - \mu) \mathrm{E}.$$
and (\ref{asymptotic}) follows directly from Gronwall's Lemma, taking $\beta=\nu-\gamma$.

\end{document}